\documentclass[b5paper, 9pt]{article}
\usepackage{seamb}

\begin{document}
\pagestyle{myheadings}
\markboth
{\hfill{\small\rm Y. B. Jun et al.}} 
{{\small\rm Characterizations of Fuzzy Fated Filters of $R_0$-algebras
           Based on Fuzzy Points} \hfill} 

\title{Characterizations of Fuzzy Fated Filters of $R_0$-algebras
           Based on Fuzzy Points}

\bth
\author{Young Bae Jun}

\no{\smaller Department of Mathematics Education (and RINS),            Gyeongsang National University,
        Chinju 660-701, Korea.\\
Email: skywine@gmail.com}

\dzh

\author{J. Kavikumar \ and \ Muhmmad Akram}

\no{\smaller Department of Mathematics and Statistics,   Faculty of Science, Technology and Human Development, Universiti Tun Hussein Onn Malaysia,
         86400 Parit Raja, Batu Pahat, Johor, Malaysia\\
Email: kaviphd@gmail.com; makram\_69@yahoo.com}

\dzh

\no{\smallerbf AMS Mathematics Subject Classification(2000):}{\smaller\ 06F35, 03G25, 08A72  }
\flh
\no{\smallerbf Abstract.}{\smaller\  More general form of the notion of quasi-coincidence of a fuzzy point with a fuzzy subset is considered, and generalization of fuzzy fated of $R_0$-algebras is discussed.
      The notion of  an $(\in, \in\! \vee \, {q}_k)$-fuzzy fated filter  in
      a $R_0$-algebra  is introduced, and several properties are
investigated. Characterizations of an $(\in, \in\! \vee \,
{q}_k)$-fuzzy fated filter in an $R_0$-algebra are discussed.
       Using a collection of fated filters,  a $(\in, \in\! \vee \, {q}_k)$-fuzzy fated filter
       is established.}
\zyh
\no{\smallerbf Keywords:}{\smaller\ (Fated) filter,   Fuzzy (fated) filter,
    $(\in,$ $\in\! \vee \, {q})$-fuzzy (fated) filter,
    $(\in,$ $\in\! \vee \, {q}_k)$-fuzzy fated filter,
    Strong $(\in,$ $\in\! \vee \, {q}_k)$-fuzzy fated filter.}

\def\iff{if and only if }
\def\TFAE{the following assertions are equivalent: }
\def\1k2{\tfrac{1-k}{2}}
\def\qk{\, {q}_k\, }
\def\ivq{\in\! \vee \, {q}\, }
\def\iwq{\in\! \wedge \, {q}\, }
\def\ivqk{\in\! \vee \, {q}_k\, }
\def\ivqr{\in\! \vee \, {q}_r\, }
\def\iwqk{\in\! \wedge \, {q}_k\, }
\def\oivqk{ \, \overline{\in}  \vee  \overline{{q}_k} \,}
\def\oivq{ \, \overline{\in}  \vee  \overline{q} \,}
\def\oi{ \, \overline{\in} \, }
\def\oqk{ \, \overline{{q}_k} \, }
\def\map{\mu}      \def\mbp{\nu}    \def\mcp{\gamma}
\def\ff{fuzzy fated }
\zjq
\normalsize
\zhangjie{Introduction}
\no
One important task of artificial intelligence is to make the
computers simulate beings in dealing with certainty and
uncertainty in information. Logic appears in a ``sacred''
(respectively, a ``profane'') form which is dominant in proof
theory (respectively, model theory). The role of logic in
mathematics and computer science is twofold -- as a tool for
applications in both areas, and a technique for laying the
foundations. Non-classical logic including many-valued logic,
fuzzy logic, etc., takes the advantage of classical logic to
handle information with various facets of uncertainty (see
\cite{INS172-1} for generalized theory of uncertainty), such as
fuzziness, randomness etc. Non-classical logic has become a formal
and useful tool for computer science to deal with fuzzy
information and uncertain information. Among all kinds of
uncertainties, incomparability is an important one which can be
encountered in our life.
        The concept of $R_0$-algebras was first introduced by Wang in
\cite{Wang00} by providing an algebraic proof of the completeness
theorem of a formal deductive system \cite{INS117-47}. Obviously,
 $R_0$-algebras are different from the BL-algebras. Further, Pei and Wang \cite{SCE32-56} proved $NM$-algebras are categorically isomorphic to $R_0$-algebras. Jun and Liu \cite{ID93249} studied (fated) filters of
$R_0$-algebras. They mentioned that the theory of $R_0$-algebras becomes one of the theoretical applications to the development of the theory of $MTL$-algebras. Some concrete practical and theoretical applications of $R_0$-algebras can be found in \cite{SCE32-56, Wang00}. Pei \cite{PEI01} proposed a new kind of fuzzy algebraic structure with the purpose to extending the concept of $R_0$-algebras and BL-algebras using results of normal residuated lattices into fuzzy settings. 
 Liu and Li \cite{INS171-61} discussed the fuzzy
set theory of filters in $R_0$-algebras.

The
algebraic structure of set theories dealing with uncertainties has
been studied by some authors. The most appropriate theory for
dealing with uncertainties is the theory of fuzzy sets developed
by Zadeh \cite{IC8-338}. 
Murali \cite{INS158-277} proposed a definition of a fuzzy point
belonging to fuzzy subset under a natural equivalence on fuzzy
subset. The idea of quasi-coincidence of a fuzzy point with a fuzzy
set, which is mentioned in \cite{JMAA76-571}, played a vital role
to generate some different types of fuzzy subsets. It is worth
pointing out that Bhakat and Das \cite{FSS51-235, FSS80-359}
initiated the concepts of $(\alpha,\beta)$-fuzzy subgroups by
using the ``belongs to'' relation $(\in \, )$ and
``quasi-coincident with'' relation $({q})$ between a fuzzy point
and a fuzzy subgroup, and introduced the concept of an $(\in,
\ivq)$-fuzzy subgroup.  In particular, an $(\in, \in\!\vee \,
{q})$-fuzzy subgroup is an important and useful generalization of
Rosenfeld's fuzzy subgroup. As a generalization of the notion of fuzzy
filters in $R_0$-algebras, Ma et al. \cite{MLQ55-493} dealt with
the notion of $(\in,$ $\ivq)$-fuzzy filters in $R_0$-algebras.
   In \cite{ID 980315}, Han et al. dealt with  the fuzzy set theory of fated filters
in $R_0$-algebras. They provided conditions for a fuzzy filter to
be a fuzzy fated filter, and introduced the notion of $(\in,$
$\ivq)$-fuzzy fated filters. They established a relation between an
$(\in,$ $\ivq)$-fuzzy filter and an $(\in,$ $\ivq)$-fuzzy fated
filter, and provided conditions for an $(\in,$ $\ivq)$-fuzzy
filter to be an $(\in,$ $\ivq)$-fuzzy fated filter. They also
dealt with characterizations of an $(\in,$ $\ivq)$-fuzzy fated
filter. It is now natural to investigate more general type of $(\in,$
$\ivq)$-\ff filters of an $R_0$-algebra.  As a first step in this
direction, we introduce the concept of an $(\in,$ $\ivqk)$-\ff
filter of an $R_0$-algebra, and discuss some fundamental aspects
of $(\in,$ $\ivqk)$-\ff filters. We deal with characterizations of
$(\in,$ $\ivqk)$-\ff filters.
   Using a collection of fated filters, we establish an $(\in,$
$\ivqk)$-\ff filter.

 The important achievement of the study with an $(\in,$
$\ivqk)$-\ff filter is that the notion of an $(\in,$ $\ivq)$-\ff
filter is a special case of an $(\in,$ $\ivqk)$-\ff filter, and
thus so many results in the paper \cite{ID 980315} are corollaries
of our results obtained in this paper.
\zjq
\zhangjie{Preliminaries}

\no
Let $L$ be a bounded distributive lattice with order-reversing
involution $\neg$ and a binary operation $\to.$ Then $(L, \wedge,
\vee, \neg, \to)$ is called an {\it $R_0$-algebra} (see
\cite{Wang00}) if it satisfies the following axioms:
\begin{lieju}
 \item[\rm (R1)] $x\to y=\neg y\to \neg x,$
 \item[\rm (R2)] $1\to x=x,$
 \item[\rm (R3)] $(y\to z)\wedge ((x\to y)\to (x\to z))=y\to z,$
 \item[\rm (R4)] $x\to (y\to z)=y\to (x\to z),$
 \item[\rm (R5)] $x\to (y\vee z)=(x\to y)\vee (x\to z),$
 \item[\rm (R6)] $(x\to y)\vee ((x\to y)\to (\neg x\vee y))=1.$
\end{lieju}

Let $L$ be an $R_0$-algebra. For any $x,y\in L,$ we define $x\odot
y=\neg (x\to \neg y)$ and $x\oplus y=\neg x\to y.$ It is proved
that $\odot$ and $\oplus$ are commutative, associative and
$x\oplus y=\neg(\neg x\odot \neg y),$ and $(L, \wedge, \vee,
\odot, \to, 0,1)$ is a residuated lattice.

For any elements $x,$ $y$ and $z$ of an $R_0$-algebra $L,$ we have
the following properties (see {\rm \cite{SCE32-56}}).
\begin{lieju}
 \item[\rm (a1)] $x\le y$ \iff $x\to y=1,$
 \item[\rm (a2)] $x\le y\to x,$
 \item[\rm (a3)] $\neg x=x\to 0,$
 \item[\rm (a4)] $(x\to y)\vee (y\to x)=1,$
 \item[\rm (a5)] $x\le y$ implies $y\to z \le x\to z,$
 \item[\rm (a6)] $x\le y$ implies $z\to x \le z\to y,$
 \item[\rm (a7)] $((x\to y)\to y)\to y=x\to y,$
 \item[\rm (a8)] $x\vee y=((x\to y)\to y)\wedge ((y\to x)\to x),$
 \item[\rm (a9)] $x\odot \neg x=0$ and $x\oplus \neg x=1,$
 \item[\rm (a10)] $x\odot y\le x\wedge y$ and $x\odot (x\to y)\le
                   x\wedge y,$
 \item[\rm (a11)] $(x\odot y)\to z=x\to (y\to z),$
 \item[\rm (a12)] $x\le y\to (x\odot y),$
 \item[\rm (a13)] $x\odot y\le z$ \iff $x\le y\to z,$
 \item[\rm (a14)] $x\le y$ implies $x\odot z\le y\odot z,$
 \item[\rm (a15)] $x\to y\le (y\to z)\to (x\to z),$
 \item[\rm (a16)] $(x\to y)\odot (y\to z)\le x\to z.$
\end{lieju}

 A non-empty subset $A$ of an $R_0$-algebra $L$ is called a {\it
  filter} of $L$ if it satisfies the following two conditions:
\begin{lieju}
  \item[\rm (b1)] $1 \in  A.$
  \item[\rm (b2)] $(\forall x \in  A)~(\forall y \in L)
        ~(x \to  y \in  A \Longrightarrow  y \in A).$
 \end{lieju}
     It can
be easily verified that a non-empty subset $A$ of an $R_0$-algebra
$L$ is a filter of $L$ if and only if it satisfies the following
conditions:
  \begin{lieju}
 \item[\rm (b3)] $(\forall x, y \in A)~(x\odot y \in A).$
 \item[\rm (b4)] $(\forall y \in L)~(\forall x \in  A)~(x \le  y \Longrightarrow  y \in A).$
\end{lieju}

\begin{definition}\label{{D091006}}
 A fuzzy subset $\map$ of an $R_0$-algebra $L$ is called a {\it fuzzy
 filter} of $L$ if it satisfies:
 \begin{lieju}
 \item[\rm (c1)] $(\forall x, y \in L)~(\map(x\odot y) \ge \min\{\map(x), \map(y)\}).$
 \item[\rm (c2)] $\map$ is order-preserving, that is,
    $(\forall x, y \in L)~(x \le  y \Longrightarrow  \map(x) \le \map(y)).$
 \end{lieju}\end{definition}
\fudl

\begin{thm}\label{{T091006}}
  A fuzzy subset $\map$ of an $R_0$-algebra $L$ is a fuzzy filter of $L$
 \iff  it satisfies:
\begin{lieju}
 \item[\rm (c3)] $(\forall x \in L)~(\map(1) \ge \map(x)),$
 \item[\rm (c4)] $(\forall x, y \in L)~(\map(y) \ge  \min\{\map(x \to y), \map(x)\}).$
\end{lieju}\end{thm}
\fudl
 For any fuzzy subset $\map$ of $L$ and $t \in  (0, 1],$ the set
 \[U(\map; t) = \{x \in L \mid \map(x) \ge t\}\] is called a
 {\it level subset} of $L.$
It is well known that a fuzzy subset $\map$ of $L$ is a fuzzy
filter of $L$ \iff the non-empty level subset $U(\map;t),$ $t\in
(0,1],$ of $\map$ is a filter of $L.$

 A fuzzy subset $\map$ of a set $L$ of the form
  \begin{equation*}
   \map(y):=\left\{\begin{array}{ll}
         t\in (0,1] &{\rm if}\;\, y=x, \\
         0 &{\rm if}\;\, y\ne x,
 \end{array}\right.
 \end{equation*}
is said to be a  {\it fuzzy point} with support $x$ and value $t$
and is denoted by $x_t.$

For a fuzzy point $x_t$ and a fuzzy subset $\map$ of $L,$ Pu and
Liu \cite{JMAA76-571} introduced  the symbol $x_t\, {\alpha}\,
\map,$ where $\alpha \in \{\in, {q}, \ivq\}.$
   We say that
\begin{lieju}
 \item[\rm (i)] $x_t$ {\it belong to} $\map,$ denoted by $x_t\in \map,$ if $\map(x)\ge t,$
 \item[\rm (ii)] $x_t$ is {\it  quasi-coincident with} $\map,$ denoted by
     $x_t\, {q}\, \map,$ if $\map(x)+t>1,$
 \item[\rm (iii)]  $x_t\ivq \map$ if $x_t\in \map$ or $x_t\, {q}\, \map,$
\end{lieju}

\zjq
\zhangjie{Generalizations of fuzzy fated filters based on fuzzy points}
\no

In what follows, $L$ is an $R_0$-algebra  unless otherwise
specified.
    In \cite{ID93249}, the notion of a fated filter of $L$ is
introduced as follows.

 A non-empty subset $A$ of  $L$ is called a {\it
  fated filter} of $L$ (see \cite{ID93249})  if it satisfies (b1) and
  \begin{equation}\begin{split}\label{{z091220}}
  (\forall x,y \in  L)~(\forall a \in A)
        ~(a\to ((x \to  y)\to x) \in  A \Longrightarrow  x \in A).
 \end{split}\end{equation}
\fudl
\begin{lem}[\cite{ID93249}]\label{{L091209}}
A filter $F$ of $L$ is fated \iff the following assertion is
valid.
 \begin{equation}\begin{split}\label{{z091209}}
 (\forall x,y,z\in L)~ \bigl(x\to (y\to z)\in F, ~x\to y\in F
        ~\Rightarrow ~x\to z\in F\bigr).
 \end{split}\end{equation}
\end{lem}
\fudl
\begin{lem}[\cite{ID93249}]\label{{L091209-1}}
A filter $F$ of $L$ is fated \iff the following assertion is
valid.
 \begin{equation}\begin{split}\label{{z091209-1}}
 (\forall x,y\in L)~ \bigl((x\to y)\to x\in F
        ~\Rightarrow ~x\in F\bigr).
 \end{split}\end{equation}
\end{lem}
\fudl

Denote by $FF(L)$ the set of all fated filters of $L.$ Note that
$FF(L)$ is a complete lattice under the set inclusion with the
largest element $L$ and the least element $\{1\}.$

In what follows,  let $k$ denote an arbitrary
 element of $[0,1)$ unless otherwise specified.
        To say that  $x_t\, \qk \, \map,$ we mean $\map(x)+t+k>1.$
        To say that $x_t\ivqk \, \map,$ we mean
$x_t\in \map$ or $x_t\, \qk \, \map.$

\begin{definition}\label{{D091115-33}}
A fuzzy subset $\map$ of  $L$ is said to be an {\it $(\in,
\ivqk)$-\ff filter} of $L$ if it satisfies:
 \begin{enumerate}
 \item[\rm (1)]  $x_t\in \map
     ~\Longrightarrow ~1_t\ivqk \map,$
 \item[\rm (2)] $(a\to ((x\to y)\to x))_t\in \map, ~a_s\in \map
    ~\Longrightarrow  ~x_{\min\{t,s\}}\ivqk \map$
 \end{enumerate}
 for all $x,a,y\in L$ and $t,s\in (0,1].$
 \end{definition}

If a fuzzy subset $\map$ of $L$ satisfies (c3) and Definition
\ref{{D091115-33}}3.3(2), then we say that $\map$ is a {\it strong
$\left(\in, \ivqk\right)$-\ff filter} of $L.$
    A (strong) $\left(\in, \ivqk\right)$-\ff filter of $L$ with
    $k=0$ is called a (strong) $\left(\in, \ivq\right)$-\ff
    filter of $L.$

\begin{example}\label{{E091003-33}}
 Let $L=\{0,a,b,c,d,1\}$ be a set with the following Hasse diagram and Cayley
 tables:

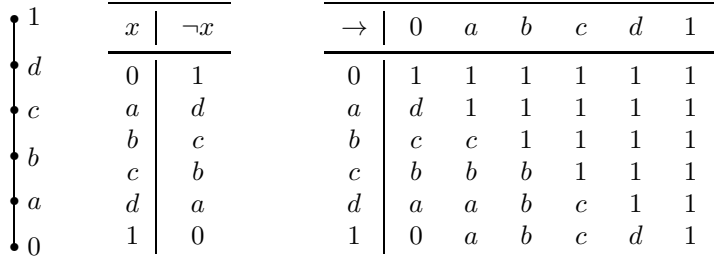
\begin{table}[h]
\begin{center}
\begin{picture}(40,30) \thinlines
  \put(20,45){\circle*{3}}
  \put(20,28){\circle*{3}}
  \put(20,11){\circle*{3}}
  \put(20,-6){\circle*{3}}
  \put(20,-23){\circle*{3}}
  \put(20,-40){\circle*{3}}
  \put(25,-43){$0$}
  \put(25,-26){$a$}
  \put(25,-10){$b$}
  \put(25,8){$c$}
  \put(25,24){$d$}
  \put(25,43){$1$}
       \put(20,-40){\line(0,1){85}}
 \end{picture} \hspace{3mm}
\begin{tabular}{c|c} \hline\noalign{\smallskip}
          $x$ &  ~$\neg x$ \\ 
   \noalign{\smallskip}\hline\noalign{\smallskip}
          $0$ & ~$1$  \\
          $a$ & ~$d$  \\
          $b$ & ~$c$  \\
          $c$ & ~$b$  \\
          $d$ & ~$a$  \\
          $1$ & ~$0$  \\
 \noalign{\smallskip}\hline
 \end{tabular} \hspace{10mm}
\begin{tabular}{c|cccccc} \hline\noalign{\smallskip}
        $\to$ & ~$0$ & ~$a$ & ~$b$ & ~$c$ & ~$d$ & ~$1$\\ 
   \noalign{\smallskip}\hline\noalign{\smallskip}
          $0$ & ~$1$ & ~$1$ & ~$1$ & ~$1$ & ~$1$ & ~$1$ \\
          $a$ & ~$d$ & ~$1$ & ~$1$ & ~$1$ & ~$1$ & ~$1$ \\
          $b$ & ~$c$ & ~$c$ & ~$1$ & ~$1$ & ~$1$ & ~$1$ \\
          $c$ & ~$b$ & ~$b$ & ~$b$ & ~$1$ & ~$1$ & ~$1$ \\
          $d$ & ~$a$ & ~$a$ & ~$b$ & ~$c$ & ~$1$ & ~$1$ \\
          $1$ & ~$0$ & ~$a$ & ~$b$ & ~$c$ & ~$d$ & ~$1$ \\
 \noalign{\smallskip}\hline
 \end{tabular}
\end{center}
\end{table}

\noindent
    Then $(L,\wedge,\vee,\neg,\to,0,1)$ is an $R_0$-algebra
(see \cite{INS171-61}), where $x\wedge y=\min\{x,y\}$ and $x\vee
y=\max\{x,y\}.$
  Define a fuzzy subset $\map$ of $L$ by
 \[\map :L\rightarrow [0,1],
   ~~x \mapsto   \begin{cases}
          0.7  & \mbox{\rm if } \, x=1, \\
          0.6  & \mbox{\rm if } \, x=c, \\
          0.4  & \mbox{\rm if } \, x=d, \\
          0.2  & \mbox{\rm if } \, x\in \{0,a,b\}.
 \end{cases} \]
It is routine to verify that $\map$ is a strong
 $(\in,$ $\in\! \vee  {q}_{0.4})$-\ff filter of $L.$
 A fuzzy subset $\mbp$ of $L$ defined by
 \[\mbp :L\rightarrow [0,1],
   ~~x \mapsto   \begin{cases}
          0.8  & \mbox{\rm if } \, x\in \{c,d\}, \\
          0.7  & \mbox{\rm if } \, x=1, \\
          0.3  & \mbox{\rm if } \, x\in \{0,a,b\}.
 \end{cases} \]
is an $(\in,$ $\in\! \vee  {q}_{0.2})$-\ff filter of $L,$ but it
is not a strong $(\in,$ $\in\! \vee  {q}_{0.2})$-\ff filter of
$L.$
   \end{example}

\begin{thm}\label{{T110902}}
 Every $(\in, \in)$-\ff filter of $L$ is an $(\in,
 \ivqk)$-\ff filter.
 \end{thm}

\begin{proof} Straightforward. \end{proof}

Taking $k=0$ in Theorem 3.5 \ref{{T110902}}, we have the following
corollary.

\begin{cor}\label{C110902}
 Every $(\in, \in)$-\ff filter of $L$ is an $(\in,
 \ivq)$-\ff filter.
 \end{cor}

The converse of Theorem 3.5 \ref{{T110902}} is not true as seen in the
following example.

\begin{example}
  The $(\in,$ $\in\! \vee  {q}_{0.2})$-\ff filter $\mbp$ of $L$ in
  Example 3.4\ref{{E091003-33}} is not an $(\in, \in)$-\ff filter of $L.$
\end{example}

Obviously, every strong $(\in, \ivqk)$-\ff filter is an $(\in,
\ivqk)$-\ff filter, but not converse as seen in Example 3.4.
\ref{{E091003-33}}

We provided characterizations of an $(\in, \ivqk)$-\ff filter.

\begin{thm}\label{{T091115-5}}
A fuzzy subset $\map$ of  $L$ is  an $(\in, \ivqk)$-\ff filter of
$L$ \iff it satisfies the following inequalities:
 \begin{enumerate}
 \item[\rm (1)] $(\forall x\in L)$ $(\map(1)\ge \min\{\map(x), \1k2\}),$
 \item[\rm (2)] $(\forall x,a,y\in L)$
   $(\map(x)\ge \min\{\map(a\to ((x\to y)\to x)), \map(a), \1k2\}).$
 \end{enumerate}
\end{thm}

\begin{proof} Let $\map$ be an $(\in, \ivqk)$-\ff
filter of $L.$ Assume that there exists $a\in L$ such that
$\map(1)<\min\{\map(a), \1k2\}.$ Then $\map(1)<t\le \min\{\map(a),
\1k2\}$ for some $t\in (0,\1k2],$ and so $a_t\in \map.$ It follows
from Definition \ref{{D091115-33}}3.3(1) that $1_t\ivqk \map,$ i.e.,
$1_t\in \map$ or $1_t\qk \map$ so that $\map(1)\ge t$ or
$\map(1)+t+k>1.$ This is a contradiction. Hence
 $\map(1)\ge \min\{\map(x), \1k2\}$ for all $x\in L.$
Suppose that there exist $a,b,c\in L$ such that
 \[\map(b)<\min\left\{\map(a\to ((b\to c)\to b)), \map(a), \1k2\right\}.\]
 Then $\map(b)<s\le \min\left\{\map(a\to ((b\to c)\to b)), \map(a), \1k2\right\}$
for some $s\in \left(0,\1k2\right].$ Thus
   $(a\to ((b\to c)\to b))_s\in \map$ and $a_s\in \map.$ Using
Definition \ref{{D091115-33}}3.3(2), we have
 $b_s=b_{\min\left\{s,s\right\}}\ivqk \map,$ which implies that
 $\map(b)\ge s$ or $\map(b)+s+k>1.$ This is a contradiction, and
 therefore
 \[\map(x)\ge \min\left\{\map(a\to ((x\to y)\to x)), \map(a), \1k2\right\}\]
for all $x,a,y\in L.$

 Conversely, let $\map$ be a fuzzy subset of
$L$ that satisfies two conditions (1) and (2). Let $x\in L$ and
$t\in (0,1]$ be such that $x_t\in \map.$ Then $\map(x)\ge t,$
which implies from (1) that  $\map(1)\ge \min\{\map(x), \1k2\}\ge
\min\{t, \1k2\}.$  If $t\le \1k2,$ then $\map(1)\ge t,$ i.e.,
$1_t\in \map.$ If $t>\1k2,$ then $\map(1)\ge \1k2$ and so
$\map(1)+t+k>\1k2+\1k2+k=1,$ i.e., $1_t\qk \map.$ Hence $1_t\ivqk
\map.$  Let $x,a,y\in L$ and $t,s\in (0,1]$ be such that $(a\to
((x\to y)\to x))_t\in \map$ and $a_s\in \map.$ Then $\map(a\to
((x\to y)\to x))\ge t$ and $\map(a)\ge s.$ It follows from (2)
that
 \begin{equation*}\begin{split}
 \map(x) &\ge \min\{\map(a\to ((x\to y)\to x)), \map(a), \1k2\}\\
  &\ge \min\{t,s,\1k2\}.
 \end{split}\end{equation*}
If $\min\{t,s\}\le \1k2,$ then $\map(x)\ge \min\{t,s\},$ which
shows that $x_{\min\{t,s\}}\in \map.$ If $\min\{t,s\}>\1k2,$ then
$\map(x)\ge \1k2,$ and thus $\map(x)+\min\{t,s\}+k>1,$ i.e.,
$x_{\min\{t,s\}}\qk \map.$ Hence $x_{\min\{t,s\}}\ivqk \map.$
Consequently, $\map$ is  an $(\in, \ivqk)$-\ff filter of $L.$
\end{proof}
\fudl
\begin{cor}
 If $\map$ is an $\left(\in, \ivqk\right)$-\ff filter of $L$ with
 $\map(1)<\1k2,$ then $\map$ is an $(\in,\in)$-\ff filter of $L.$
 \end{cor}
\fudl
\begin{cor}[\cite{ID 980315}]\label{{C110901}}
A fuzzy subset $\map$ of  $L$ is  an $(\in, \ivq)$-\ff filter of
$L$ \iff it satisfies the following inequalities:
 \begin{enumerate}
 \item[\rm (1)] $(\forall x\in L)$ $(\map(1)\ge \min\{\map(x), 0.5\}),$
 \item[\rm (2)] $(\forall x,a,y\in L)$
   $(\map(x)\ge \min\{\map(a\to ((x\to y)\to x)), \map(a), 0.5\}).$
 \end{enumerate}
\end{cor}

\begin{proof} It follows from taking $k=0$ in Theorem 3.8.
\ref{{T091115-5}}\end{proof}
\fudl
\begin{cor}\label{{C091123-3}}
Every strong $(\in,$ $\ivqk)$-\ff filter $\map$ of $L$ satisfies
the following inequalities:
 \begin{enumerate}
  \item[\rm (1)] $(\forall x\in L)~\bigl(\map(1)\ge \min\{\map(x), \1k2\}\bigr),$
  \item[\rm (2)] $(\forall x,a,y\in L)
    ~\bigl(\map(x)\ge \min\{\map(a\to ((x\to y)\to x)), \map(a), \1k2\}\bigr).$
\end{enumerate}\end{cor}
\fudl
\begin{thm}\label{{T091115-66}}
A fuzzy subset $\map$ of  $L$ is  an $(\in, \ivqk)$-\ff filter of
$L$ \iff it satisfies the following assertion:
  \begin{eqnarray}\label{{z091115-6}}
 \left(\forall t\in \left(0,\1k2\right]\right)
   ~\left(U(\map;t)\in FF(L)\cup \{\emptyset\}\right).
 \end{eqnarray}\end{thm}
 
\begin{proof} Assume that $\map$ is  an $\left(\in, \ivqk\right)$-\ff
filter of $L.$ Let $t\in \left(0,\1k2\right]$ be such that
$U(\map;t)\ne \emptyset.$ Then there exists $x\in U(\map;t),$ and
so $\map(x)\ge t.$ Using Theorem \ref{{T091115-5}}3.8(1), we get
 \[\map(1)\ge \min\left\{\map(x), \1k2\right\}\ge \min\left\{t,\1k2\right\}=t,\]
i.e., $1\in U(\map;t).$ Assume that  $a\to ((x\to y)\to x)\in
U(\map;t)$ for all $x,y\in L$ and $a\in U(\map;t).$   Then
 $\map(a\to ((x\to y)\to x))\ge t$ and $\map(a)\ge t.$ It follows
 from Theorem \ref{{T091115-5}}3.8(2) that
 \begin{equation*}\begin{split}
 \map(x) &\ge \min\left\{\map(a\to ((x\to y)\to x)), \map(a), \1k2\right\}\\
 &\ge \min\left\{t,\1k2\right\}=t
 \end{split}\end{equation*}
 so that $x\in U(\map;t).$ Therefore $U(\map;t)$ is a fated filter
 of $L.$

Conversely, let $\map$ be a fuzzy subset of $L$ satisfying the
assertion (\ref{{z091115-6}}). Assume that
  $\map(1)<\min\left\{\map(a), \1k2\right\}$ for some $a\in L.$ Putting
  $t_a:=\min\left\{\map(a), \1k2\right\},$ we have $a\in U\left(\map;t_a\right)$ and so
  $U\left(\map;t_a\right)\ne \emptyset.$ Hence $U\left(\map;t_a\right)$ is a fated
  filter of $L$ by (\ref{{z091115-6}}), which implies that $1\in
U\left(\map;t_a\right).$ Thus $\map(1)\ge t_a,$ which is a
contradiction. Therefore $\map(1)\ge \min\left\{\map(x),
\1k2\right\}$ for all $x\in L.$ Suppose that
 \[\map(x)<\min\left\{\map(a\to ((x\to y)\to x)), \map(a), \1k2\right\}\]
for some $x,a,y\in L.$ Taking
 $t_x:=\min\left\{\map(a\to ((x\to y)\to x)), \map(a), \1k2\right\},$ we get
 $a\to ((x\to y)\to x)\in U\left(\map;t_x\right)$ and $a\in U\left(\map;t_x\right).$
It follows from (\ref{{z091220}}) that $x\in
U\left(\map;t_x\right),$ i.e., $\map(x)\ge t_x.$ This is a
contradiction. Hence
 \[\map(x)\ge \min\left\{\map(a\to ((x\to y)\to x)), \map(a), \1k2\right\}\]
for all $x,a,y\in L.$ Using Theorem 3.8\ref{{T091115-5}}, we conclude
that $\map$ is  an $\left(\in, \ivqk\right)$-\ff filter of $L.$
\end{proof}

If we take $k=0$ in Theorem 3.12\ref{{T091115-66}}, then we have the
following corollary.

\begin{cor}[\cite{ID 980315}] \label{{C091115-66}}
A fuzzy subset $\map$ of  $L$ is  an $(\in, \ivq)$-\ff filter of
$L$ \iff it satisfies the following assertion:
  \begin{eqnarray}\label{{z091115-66}}
 (\forall t\in (0,0.5])~\left(U(\map;t)\in FF(L)\cup \{\emptyset\}\right).
 \end{eqnarray}\end{cor}

\begin{thm}\label{{T110923}}
 If $k<r$ in $[0,1),$ then every $\left(\in, \ivqk\right)$-fuzzy
 \ff filter of $L$ is an $\left(\in, \ivqr\right)$-fuzzy
 \ff filter.
 \end{thm}

 \begin{proof} Straightforward. \end{proof}

 The converse of Theorem 3.14\ref{{T110923}} is not true as seen in
 the following example.

 \begin{example}
  Consider an $R_0$-algebra $L=\{0,a,b,c,d,1\}$ which is appeared in
 Example 3.4\ref{{E091003-33}}.
  Define a fuzzy subset $\map$ of $L$ by
 \[\map :L\rightarrow [0,1],
   ~~x \mapsto   \begin{cases}
          0.9  & \mbox{\rm if } \, x=d, \\
          0.7  & \mbox{\rm if } \, x=c, \\
          0.3  & \mbox{\rm if } \, x=1, \\
          0.1  & \mbox{\rm if } \, x\in \{0,a,b\}.
 \end{cases} \]
It is routine to verify that $\map$ is an
 $(\in,$ $\in\! \vee  {q}_{0.4})$-\ff filter of $L.$
 Since
  \[U(\map;t)=\begin{cases}
          \{c,d\}  & \mbox{\rm if } \, t\in (0.3, 0.35], \\
          \{c,d,1\}  & \mbox{\rm if } \, t\in (0.1, 0.3], \\
          L  & \mbox{\rm if } \, t\in (0, 0.1], \\
 \end{cases} \]
we know from Theorem 3.12\ref{{T091115-66}} that $\map$ is not
 an $(\in,$ $\in\! \vee  {q}_{0.3})$-\ff filter of $L.$
 \end{example}

\begin{prop}\label{{P091209-9}}
 Every $(\in,$ $\ivqk)$-\ff filter $\map$ of $L$ satisfies the
 following inequalities.
 \begin{enumerate}
 \item[\rm (1)] $\map(x\to z)\ge \min\left\{\map(x\to (y\to z)),
                \map(x\to y), \1k2\right\},$
 \item[\rm (2)] $\map(x)\ge \min\left\{\map((x\to y)\to x), \1k2\right\}$
\end{enumerate}
for all $x,y,z\in L.$
\end{prop}

\begin{proof} (1) Suppose that there exist $a,b,c\in L$ such that
 \[\map(a\to c)<\min\left\{\map(a\to (b\to c)), \map(a\to b), \1k2\right\}.\]
Taking
 $t:=\min\left\{\map(a\to (b\to c)), \map(a\to b),\1k2\right\}$
  implies that
   $a\to (b\to c)\in U(\map;t),$  $a\to b\in U(\map;t)$ and $t\in \left(0,\1k2\right].$ Since
$U(\map;t)\in FF(L)$ by Theorem 3.12\ref{{T091115-66}}, it follows
from Lemma 3.1\ref{{L091209}} that $a\to c\in U(\map;t),$ i.e.,
$\map(a\to c)\ge t.$ This is a contradiction, and therefore $\map$
satisfies (1).

(2) If $\map$ is an $\left(\in,\ivqk\right)$-\ff filter of $L,$
then $U(\map;t)\in FF(L)\cup \{\emptyset\}$ for all $t\in
\left(0,\1k2\right]$ by Theorem 3.12\ref{{T091115-66}}. Hence
$U(\map;t)\in F(L)\cup \{\emptyset\}$ for all $t\in
\left(0,\1k2\right].$ Suppose that
 \[\map(x)<t\le \min\left\{\map((x\to y)\to x), \1k2\right\}\]
for some $x,y\in L$ and $t\in \left(0,\1k2\right].$  Then $(x\to
y)\to x\in U(\map;t),$ which implies from Lemma 3.2\ref{{L091209-1}}
that $x\in U(\map;t),$ i.e., $\map(x)\ge t.$ This is a
contradiction. Hence
 $\map(x)\ge \min\left\{\map((x\to y)\to x), \1k2\right\}$ for all $x,y\in L.$
\end{proof}
\fudl
\begin{cor}[\cite{ID 980315}]\label{{C091209}}
 Every $(\in,$ $\ivq)$-\ff filter $\map$ of $L$ satisfies the
 following inequalities.
 \begin{enumerate}
 \item[\rm (1)] $\map(x\to z)\ge \min\{\map(x\to (y\to z)),
                \map(x\to y), 0.5\},$
 \item[\rm (2)] $\map(x)\ge \min\{\map((x\to y)\to x), 0.5\}$
\end{enumerate}
for all $x,y,z\in L.$
\end{cor}
\fudl
\begin{thm}\label{{T091118-8}}
 If $F$ is a fated filter of $L,$ then a fuzzy subset $\map$ of
 $L$ defined by
 \[\map:L\rightarrow [0,1], ~x\mapsto
   \left\{\begin{array}{ll}
         t_1 & \textrm{if \, $x\in F$}, \\
         t_2 & \textrm{otherwise}
 \end{array}\right.\]
where $t_1\in \left[\1k2,1\right]$ and $t_2\in
\left(0,\1k2\right),$ is an $\left(\in, \ivqk\right)$-\ff filter
of $L.$
\end{thm}

\begin{proof} Note that
  \begin{equation*}
   U(\map;s)=\left\{\begin{array}{ll}
         F & \textrm{if \, $s\in (t_2, \1k2]$}, \\
         L & \textrm{if \, $s\in (0,t_2]$}
 \end{array}\right.
 \end{equation*}
which is a fated filter of $L.$ It follows from Theorem 3.12
\ref{{T091115-66}} that $\map$ is an $\left(\in, \ivqk\right)$-\ff
filter of $L.$
\end{proof}
\fudl
\begin{cor}[\cite{ID 980315}]\label{{C091118}}
 If $F$ is a fated filter of $L,$ then a fuzzy subset $\map$ of
 $L$ defined by
 \[\map:L\rightarrow [0,1], ~x\mapsto
   \left\{\begin{array}{ll}
         t_1 & \textrm{if \, $x\in F$}, \\
         t_2 & \textrm{otherwise}
 \end{array}\right.\]
where $t_1\in [0.5,1]$ and $t_2\in (0,0.5),$ is an $(\in,
\ivq)$-\ff filter of $L.$
\end{cor}
For any fuzzy subset $\map$ of $L$ and any $t\in (0,1],$ we
consider two subsets:
 \begin{equation*}\begin{split}
  & Q(\map;t):=\left\{x\in L\mid x_t\, {q}\, \map\right\},
         ~~[\map]_t:=\left\{x\in L\mid x_t\ivq \, \map\right\}.
\end{split} \end{equation*}
It is clear that $[\map]_t=U(\map;t)\cup Q(\map;t)$ (see \cite{ID
980315}). We also consider the following two sets:
 \[
   Q_k(\map;t):=\left\{x\in L\mid x_t\qk \map\right\},
         ~~[\map]_t^k:=\left\{x\in L\mid x_t\ivqk \, \map\right\}.\]
Obviously, $[\map]_t^k=U(\map;t)\cup Q_k(\map;t)$ and if $k=0$
then $Q_k(\map;t)=Q(\map;t)$ and $[\map]_t^k=[\map]_t.$

\begin{thm}\label{{T091124-4}}
If $\map$ is an $\left(\in, \ivqk\right)$-\ff filter of $L,$ then
 \[\left(\forall t\in (\1k2,1]\right)~\left(Q_k(\map;t)\in FF(L)\cup \{\emptyset\}\right).\]
\end{thm}

\begin{proof} Assume that $\map$ is an $\left(\in, \ivqk\right)$-\ff filter of $L$
and let $t\in \left(\1k2,1\right]$ be such that
  $Q_k(\map;t)\ne \emptyset.$  Then there exists $x\in Q_k(\map;t),$ and so
$\map(x)+t+k>1.$ Using Theorem 3.8\ref{{T091115-5}}(1), we have
\begin{equation*}\begin{split}
  \map(1)&\ge \min\left\{\map(x),\1k2\right\}  \\
   & =\left\{\begin{array}{ll}
         \1k2 & \textrm{if \, $\map(x)\ge \1k2$}, \\
         \map(x) & \textrm{if \, $\map(x)<\1k2$}
 \end{array}\right.\\
  &>1-t-k,
  \end{split}\end{equation*}
which implies that $1\in Q_k(\map;t).$ Assume that
  $a\to ((x\to y)\to x)\in Q_k(\map;t)$ and $a\in Q_k(\map;t)$ for all $x,a,y\in L.$
 Then $(a\to ((x\to y)\to x))_t\qk \map$ and $a_t\qk \map,$
that is,
 $\map(a\to ((x\to y)\to x))>1-t-k$ and $\map(a)>1-t-k.$
 Using Theorem \ref{{T091115-5}}3.8(2), we get
\[\map(x)\ge \min\left\{\map(a\to ((x\to y)\to x)), \map(a), \1k2\right\}.\]
Thus, if $\min\{\map(a\to ((x\to y)\to x)), \map(a)\}< \1k2,$ then
 \[\map(x)\ge \min\left\{\map(a\to ((x\to y)\to x)), \map(a)\right\}>1-t-k.\] If
$\min\{\map(a\to ((x\to y)\to x)), \map(a)\}\ge \1k2,$ then
$\map(x)\ge \1k2>1-t-k.$ It follows that $x_t\qk \map$ so that
$x\in Q_k(\map;t).$ Therefore $Q_k(\map;t)$ is a fated filter of
$L.$
\end{proof}
\fudl
\begin{cor}[\cite{ID 980315}]\label{{C091124-4}}
If $\map$ is an $(\in, \ivq)$-\ff filter of $L,$ then
 \[(\forall t\in (0.5,1])~\bigl(Q(\map;t)\in FF(L)\cup \{\emptyset\}\bigr).\]
\end{cor}
\fudl
\begin{cor}\label{{C091124-4}}
If $\map$ is a strong $\left(\in, \ivqk\right)$-\ff filter of $L,$
then
 \[(\forall t\in (\1k2,1])~\left(Q_k(\map;t)\in FF(L)\cup \{\emptyset\}\right).\]
\end{cor}
\fudl
The converse of Corollary \ref{{C091124-4}} is not true as shown
by the following example.

\begin{example}\label{{E091124-4}}
 Consider the $(\in,$ $\in\! \vee  {q}_{0.2})$-\ff filter $\mbp$ of $L$
 which is given in Example 3.4\ref{{E091003-33}}.
Then
     \begin{equation*}
   Q_k(\mbp;t)=\left\{\begin{array}{ll}
         L & \textrm{if \, $t\in (0.5, 1]$}, \\
         \{c,d,1\} & \textrm{if \, $t\in (0.4,0.5]$}
 \end{array}\right.
 \end{equation*}
is a fated filter of $L.$ But $\mbp$ is not a  strong $(\in,$
$\in\! \vee  {q}_{0.2})$-\ff filter of $L.$
\end{example}

\begin{thm}\label{{T091225-5}}
 For a fuzzy subset $\map$ of $L,$ the following assertions are
 equivalent:
 \begin{enumerate}
 \item[\rm (1)] $\map$ is an $\left(\in,\ivqk\right)$-\ff filter of $L.$
 \item[\rm (2)] $(\forall t\in (0,1])~\left([\map]_t^k\in FF(L)\cup
         \{\emptyset\}\right).$
 \end{enumerate}

We call $[\map]_t^k$ an {\it $(\ivqk)$-level fated filter} of
$\map.$
\end{thm}

 \begin{proof} Assume that $\map$ is an $\left(\in,\ivqk\right)$-\ff filter of
 $L$ and let $t\in (0,1]$ be such that $[\map]_t^k\ne \emptyset.$
Then  there exists $x\in [\map]_t^k=U(\map;t)\cup Q_k(\map;t),$
and so $x\in U(\map;t)$ or $x\in Q_k(\map;t).$
 If $x\in U(\map;t),$ then $\map(x)\ge t.$ It follows from Theorem
 \ref{{T091115-5}}3.8(1) that
\begin{equation*}\begin{split}
 \map(1)&\ge \min\left\{\map(x), \1k2\right\}\ge \min\left\{t,\1k2\right\} \\
   & =\left\{\begin{array}{ll}
         t & \textrm{if \, $t\le \1k2$}, \\
         \1k2>1-t-k & \textrm{if \, $t>\1k2$}
 \end{array}\right.\\
  \end{split}\end{equation*}
so that $1\in U(\map;t)\cup Q_k(\map;t)=[\map]_t^k.$ If $x\in
Q_k(\map;t),$ then  $\map(x)+t+k>1.$ Thus
\begin{equation*}\begin{split}
 \map(1)&\ge \min\{\map(x), \1k2\}\ge \min\{1-t-k,\1k2\} \\
   & =\left\{\begin{array}{ll}
         1-t-k & \textrm{if \, $t> \1k2$}, \\
         \1k2\ge t & \textrm{if \, $t\le \1k2$}
 \end{array}\right.\\
  \end{split}\end{equation*}
and so $1\in Q_k(\map;t)\cup U(\map;t)=[\map]_t^k.$
        Let $x,a,y\in L$ be such that $a\in [\map]_t^k$ and
$a\to ((x\to y)\to x)\in [\map]_t^k.$ Then
 \[ \map(a)\ge t ~\text{ or } ~\map(a)+t+k>1,\]
  and
 \[\map(a\to ((x\to y)\to x))\ge t ~\text{ or }
      ~\map(a\to ((x\to y)\to x))+t+k>1.\]
 We can consider four cases:
  \begin{eqnarray}
  && \text{\rm  $\map(a)\ge t$ and $\map(a\to ((x\to y)\to x))\ge t,$}  \label{{z091224-1}}\\
  && \text{\rm  $\map(a)\ge t$ and $\map(a\to ((x\to y)\to x))+t+k>1,$}  \label{{z091224-2}}\\
  && \text{\rm  $\map(a)+t+k>1$ and $\map(a\to ((x\to y)\to x))\ge t,$}  \label{{z091224-3}}\\
  && \text{\rm  $\map(a)+t+k>1$ and $\map(a\to ((x\to y)\to x))+t+k>1.$} \label{{z091224-4}}
\end{eqnarray}
For the first case, Theorem \ref{{T091115-5}}3.8(2) implies that
\begin{equation*}\begin{split}
 \map(x)&\ge \min\{\map(a\to ((x\to y)\to x)), \map(a), \1k2\}\\
  &\ge \min\{t,\1k2\}
    =\left\{\begin{array}{ll}
         \1k2 & \textrm{if \, $t> \1k2$}, \\
         t   & \textrm{if \, $t\le \1k2$}
 \end{array}\right.\\
  \end{split}\end{equation*}
so that $x\in U(\map;t)$ or $\map(x)+t+k>\1k2+\1k2+k=1,$ i.e.,
$x\in Q_k(\map;t).$ Hence $x\in [\map]_t^k.$
    Case (\ref{{z091224-2}}) implies that
\begin{equation*}\begin{split}
 \map(x)&\ge \min\{\map(a\to ((x\to y)\to x)), \map(a), \1k2\}\\
  &\ge \min\{1-t-k,t,\1k2\}
    =\left\{\begin{array}{ll}
         1-t-k & \textrm{if \, $t> \1k2$}, \\
         t   & \textrm{if \, $t\le \1k2.$}
 \end{array}\right.\\
  \end{split}\end{equation*}
Thus $x\in Q_k(\map;t)\cup U(\map;t)=[\map]_t^k.$ Similarly, $x\in
[\map]_t^k$ for the case (\ref{{z091224-3}}). The final case
implies that
\begin{equation*}\begin{split}
 \map(x)&\ge \min\{\map(a\to ((x\to y)\to x)), \map(a), \1k2\}\\
  &\ge \min\{1-t-k,\1k2\}
    =\left\{\begin{array}{ll}
         1-t-k & \textrm{if \, $t> \1k2$}, \\
         \1k2   & \textrm{if \, $t\le \1k2$}
 \end{array}\right.\\
  \end{split}\end{equation*}
so that $x\in Q_k(\map;t)\cup U(\map;t)=[\map]_t^k.$ Consequently
$[\map]_t^k$ is a \ff filter of $L.$

Conversely, let $\map$ be a fuzzy subset of $L$ such that
$[\map]_t^k$ is a fated filter of $L$ whenever it is nonempty for
all $t\in (0,1].$ If there exists $a\in L$ such that
$\map(1)<\min\{\map(a), \1k2\},$ then
 $\map(1)<t_a\le \min\{\map(a), \1k2\}$ for some $t_a\in (0,\1k2].$
 It follows that $a\in U(\map;t_a)$ but $1\notin U(\map;t_a).$
Also, $\map(1)+t_a+k<2t_a+k\le 1$ and so $1\notin Q_k(\map;t_a).$
Hence $1\notin U(\map;t_a)\cup Q_k(\map;t_a)=[\map]_{t_a}^k,$
which is a contradiction. Therefore $\map(1)\ge \min\{\map(x),
\1k2\}$ for all $x\in L.$ Suppose that
 \begin{equation}\begin{split}\label{{z091224-5}}
 \map(x)<\min\{\map(a\to ((x\to y)\to x)), \map(a), \1k2\}
\end{split}\end{equation}
for some $x,a,y\in L.$  Taking
 $t:=\min\{\map(a\to ((x\to y)\to x)), \map(a), \1k2\}$ implies
 that $t\in (0,\1k2],$ $a\in U(\map;t)\subseteq [\map]_t^k$ and
$a\to ((x\to y)\to x)\in U(\map;t)\subseteq [\map]_t^k.$ Since
$[\map]_t^k\in FF(L),$ it follows that $x\in
[\map]_t^k=U(\map;t)\cup Q_k(\map;t).$  But (\ref{{z091224-5}})
induces $x\notin U(\map;t)$ and $\map(x)+t+k<2t+k\le 1,$ i.e.,
$x\notin Q_k(\map;t).$ This is a contradiction, and thus
 $\map(x)\ge \min\{\map(a\to ((x\to y)\to x)), \map(a), \1k2\}$
for all $x,a,y\in L.$ Using Theorem 3.8\ref{{T091115-5}}, we conclude
that $\map$ is an $(\in,$ $\ivqk)$-\ff filter of $L.$
\end{proof}
\fudl
\begin{cor}[\cite{ID 980315}]\label{{C091225-5}}
 For a fuzzy subset $\map$ of $L,$ the following assertions are
 equivalent:
 \begin{enumerate}
 \item[\rm (1)] $\map$ is an $(\in,$ $\ivq)$-\ff filter of $L.$
 \item[\rm (2)] $(\forall t\in (0,1])~\left([\map]_t\in FF(L)\cup
         \{\emptyset\}\right).$
 \end{enumerate}\end{cor}

\begin{proof} Taking $k=0$ in Theorem 3.24\ref{{T091225-5}} induced
the desired result.
 \end{proof}
\fudl
\begin{thm}\label{{T090503-11}}
Given any chain of  fated filters $F_0\subset F_1\subset \cdots
\subset F_n=L$ of $L,$ there exists an $(\in,$ $\ivqk)$-fuzzy
fated filter $\map$ of $L$ whose level fated filters are precisely
the members of the chain with $U(\map;\tfrac{1-k}{2})=F_0.$
\end{thm}

\begin{proof} Let $\{t_i\in (0,\tfrac{1-k}{2})\mid i=1,2,\cdots, n\}$
be such that $t_1>t_2>\cdots >t_n.$ Define a fuzzy subset $\map$
of $L$ by
 \[\map : L\rightarrow [0,1],
   ~~ x\mapsto
     \left\{\begin{array}{ll}
       t_0 ~(\ge \tfrac{1-k}{2}) &{\rm if}\;\, x=1,\\
       t ~(\ge  t_0)   &{\rm if}\;\, x\in F_0\setminus \{1\},\\
       t_1     &{\rm if}\;\, x\in F_1 \setminus F_0,\\
       t_2     &{\rm if}\;\, x\in F_2 \setminus F_1,\\
       \cdots \\
       t_n     &{\rm if}\;\, x\in F_n \setminus F_{n-1}.
\end{array}\right.  \]
Then
\begin{equation*}
  U(\map;s)=\left\{\begin{array}{ll}
        F_0  &{\rm if}\;\, s\in (t_1, \tfrac{1-k}{2}],\\
        F_1  &{\rm if}\;\, s\in (t_2, t_1],\\
        F_2  &{\rm if}\;\, s\in (t_3, t_2],\\
        \cdots \\
        F_n=R  &{\rm if}\;\, s\in (0, t_n].
\end{array}\right.   \end{equation*}
Using Theorem 3.12\ref{{T091115-66}}, we know that $\map$ is an
$(\in,$ $\ivqk)$-fuzzy fated filter of $L.$ It follows from the
construction of $\map$ that $U(\map;\tfrac{1-k}{2})=F_0$ and
$U(\map;t_i)=F_i$ for $i=1,2,\cdots, n.$
\end{proof}
\fudl
\begin{cor}\label{{C090503-11}}
Given any chain of  fated filters $F_0\subset F_1\subset \cdots
\subset F_n=L$ of $L,$ there exists an $(\in,$ $\ivq)$-fuzzy fated
filter $\map$ of $L$ whose level fated filters are precisely the
members of the chain with $U(\map;0.5)=F_0.$
\end{cor}

Using a class of fated filters, we make an $(\in,$ $\ivqk)$-fuzzy
fated filter of $L.$

\begin{thm}\label{{T090503-22}}
 Let $\{F_t\mid t\in \Lambda\},$ where $\Lambda \subseteq
 \left(0,\1k2\right],$ be a collection of fated filters of $L$ such that
  \begin{enumerate}
  \item[\rm (i)] $L=\bigcup\limits_{t\in \Lambda}F_t,$
 \item[\rm (ii)] $(\forall s,t\in \Lambda)$ $(s<t \,
 \Leftrightarrow \, F_t\subset F_s).$
 \end{enumerate}
 Then a fuzzy subset $\map$ of $L$ defined by $\map(x)=\sup\{t\in
 \Lambda \mid x\in F_t\}$ for all $x\in L$ is an $(\in,$
 $\ivqk)$-fuzzy fated filter of $L.$
\end{thm}

\begin{proof} According to Theorem 3.12\ref{{T091115-66}}, it is
sufficient to show that $U(\map;t)\ne \emptyset$ is a fated filter
of $L$ for all $t\in (0,\tfrac{1-k}{2}].$ We consider two cases:
 \[\text{\rm (i)\, $t=\sup\{s\in \Lambda \mid s<t\},$ \, \,
            (ii)\, $t\ne \sup\{s\in \Lambda \mid   s<t\}.$}\]
Case (i) implies that
 \[x\in U(\map;t) \, \Longleftrightarrow \, (x\in F_s, ~~\forall
 s<t) \, \Longleftrightarrow \, x\in \bigcap\limits_{s<t}F_s,\]
and so $U(\map;t)=\bigcap\limits_{s<t}F_s$ which is a fated filter
of $L.$  In the second case, we have
$U(\map;t)=\bigcup\limits_{s\ge t}F_s.$ Indeed, if $x\in
\bigcup\limits_{s\ge t}F_s,$ then $x\in F_s$ for some $s\ge t.$
Thus $\map(x)\ge s\ge t,$ i.e., $x\in U(\map;t).$ This proves
$\bigcup\limits_{s\ge t}F_s\subset U(\map;t).$ To prove the
reverse inclusion, let $x\notin \bigcup\limits_{s\ge t}F_s.$ Then
$x\notin F_s$ for all $s\ge t.$ Since $t\ne  \sup\{s\in \Lambda
\mid   s<t\},$ there exists $\varepsilon >0$ such that
$(t-\varepsilon, t)\cap \Lambda =\emptyset.$ Hence $x\notin F_s$
for all $s>t-\varepsilon,$ which means that if $x\in F_s$ then
$s\le t-\varepsilon.$ Thus $\map(x)\le t-\varepsilon<t,$ and so
$x\notin U(\map;t).$ Therefore $U(\map;t)=\bigcup\limits_{s\ge
t}F_s$ which is also a fated filter of $L.$ Consequently, $\map$
is an $(\in,$ $\ivqk)$-fuzzy fated filter of $L.$
\end{proof}

\begin{cor}\label{{C090503-2}}
 Let $\{F_t\mid t\in \Lambda\},$ where $\Lambda \subseteq
 (0,0.5],$ be a collection of fated filters of $L$ such that
  \begin{enumerate}
  \item[\rm (i)] $L=\bigcup\limits_{t\in \Lambda}F_t,$
 \item[\rm (ii)] $(\forall s,t\in \Lambda)$ $(s<t \,
 \Leftrightarrow \, F_t\subset F_s).$
 \end{enumerate}
 Then a fuzzy subset $\map$ of $L$ defined by $\map(x)=\sup\{t\in
 \Lambda \mid x\in F_t\}$ for all $x\in L$ is an $(\in,$
 $\ivq)$-fuzzy fated filter of $L.$
\end{cor}

A fuzzy subset $\map$ of  $L$ is said to be {\it proper} if ${\rm
Im}(\map)$ has at least two elements. Two fuzzy subsets are said
to be {\it equivalent} if they have same family of level subsets.
Otherwise, they are said to be {\it non-equivalent}.

\begin{thm}\label{{T080727-7}}
Let $\map$ be an $(\in,\ivqk)$-\ff filter  of  $L$ such that
$\#\{\map(x)\mid \map(x)<\tfrac{1-k}{2}\}\ge 2.$ Then there exist
two proper non-equivalent
 $(\in,\ivqk)$-\ff filters of $L$ such that $\map$ can be expressed as the union of them.
\end{thm}

\begin{proof}  Let $\{\map(x)\mid \map(x)<\tfrac{1-k}{2}\}=\{t_1, t_2, \cdots,
t_r\},$ where $t_1>t_2>\cdots >t_r$ and $r\ge 2.$ Then the chain
of $(\ivqk)$-level fated filters of $\map$ is
 \[[\map]_{\tfrac{1-k}{2}}^k\subseteq [\map]_{t_1}^k\subseteq [\map]_{t_2}^k\subseteq \cdots
 \subseteq [\map]_{t_r}^k=L.\]
Let $\mbp$ and $\mcp$ be fuzzy subsets of $L$ defined by
  \begin{equation}
 \mbp(x)=\left\{\begin{array}{ll}
     t_1  &{\rm if}\;\, x\in [\map]_{t_1}^k,\\
     t_2  &{\rm if}\;\, x\in [\map]_{t_2}^k\setminus [\map]_{t_1}^k,\\
     \cdots  &{\rm }\;\, \\
     t_r  &{\rm if}\;\, x\in [\map]_{t_r}^k\setminus [\map]_{t_{r-1}}^k,\\
\end{array}\right. \nonumber
\end{equation}
and
  \begin{equation}
 \mcp(x)=\left\{\begin{array}{ll}
     \map(x) &{\rm if}\;\, x\in [\map]_{\1k2}^k,\\
     k   &{\rm if}\;\, x\in [\map]_{t_2}^k\setminus [\map]_{\1k2}^k,\\
     t_3  &{\rm if}\;\, x\in [\map]_{t_3}^k\setminus [\map]_{t_2}^k,\\
     \cdots  &{\rm }\;\, \\
     t_r  &{\rm if}\;\, x\in [\map]_{t_r}^k\setminus [\map]_{t_{r-1}}^k,\\
\end{array}\right. \nonumber
\end{equation}
respectively, where $t_3<k<t_2.$ Then $\mbp$ and $\mcp$ are $(\in,
\ivqk)$-fuzzy fated filters of $L,$ and $\mbp, \mcp\le \map.$ The
chains of $(\ivqk)$-level fated filters of $\mbp$ and $\mcp$ are,
respectively, given by
\[ [\map]_{t_1}^k\subseteq [\map]_{t_2}^k\subseteq \cdots
 \subseteq [\map]_{t_r}^k\]
and
 \[[\map]_{\1k2}^k\subseteq [\map]_{t_2}^k\subseteq \cdots
 \subseteq [\map]_{t_r}^k.\]
Therefore $\mbp$ and $\mcp$ are non-equivalent and clearly
$\map=\mbp\cup \mcp.$ This completes the proof.
\end{proof}

\zhangjie{Conclusion}
\no
In this paper, using the "belongs to" relation ($\in$) and quasicoincidence with relation ($q$) between the fuzzy point and fuzzy sets, we introduce the notions of $(\in, \in \vee \,{q_k})$-fuzzy fated filter in an $R_0$-algebras and investigate some related properties.  We have dealt with characterizations of an  $(\in, \in \vee \,{q_k})$-fuzzy fated filter in $R_0$-algebras and have obtained an $(\in, \in \vee \,{q_k})$-fuzzy fated filter is that the notion of an $(\in, \in \vee \,{q})$-fuzzy fated filter is a special case of an $(\in, \in \vee \,{q_k})$-fuzzy fated filter. Based on these results, we shall focus on other types and their relationships among them, and also consider these generalized rough fuzzy filters of $R_0$-algebras.
\fudl

\end{document}